\documentclass[11pt,reqno]{amsart}
\usepackage{amsmath}

\usepackage{color}
\usepackage{graphics}
\usepackage{amsmath}
\usepackage{amssymb}
\usepackage{graphicx}
\usepackage{subfigure}

\usepackage[left=25mm,right=25mm,top=25mm,bottom=25mm,a4paper]{geometry}

\makeatletter
\newcommand{\solutionname}{\textbf{Example}}

\makeatother

\theoremstyle{plain}
\newtheorem{Thm}{Theorem}
\newtheorem{Cor}[Thm]{Corollary}
\newtheorem{Lem}[Thm]{Lemma}
\newtheorem{Prop}[Thm]{Proposition}

\theoremstyle{definition}
\newtheorem{Def}{Definition}
\newtheorem{Assumption}[Def]{Assumption}
\newtheorem{Ex}[Def]{Example}

\newenvironment{ExQED}
 {\pushQED{\qed}\Ex}
 {\popQED\endEx}

\theoremstyle{remark}
\newtheorem{Rem}{Remark}

\numberwithin{equation}{section}
\numberwithin{Thm}{section}
\numberwithin{Def}{section}
\numberwithin{Rem}{section}
\def\R{\mathbf{R}}

\begin{document}
\large

\title[Efficiency Formula under Constant Seebeck Coefficients]{Unique Temperature Distribution and Explicit Efficiency Formula for One-Dimensional Thermoelectric Generators under Constant Seebeck Coefficients}

\author{Jaywan Chung}
\address[Chung]{Enegy Conversion Research Center, Korea
 Electrotechnology Research Institute, 12, Jeongiui-gil,
 Changwon-si, Gyeongsangnam-do, 51543 Korea}
\email{jchung@keri.re.kr}

\author{Byungki Ryu}
\address[Ryu]{Enegy Conversion Research Center, Korea
 Electrotechnology Research Institute, 12, Jeongiui-gil,
 Changwon-si, Gyeongsangnam-do, 51543 Korea}
\email{byungkiryu@keri.re.kr}

\author{Hyowon Seo}
\address[Seo]{Department of Applied Mathematics and the Institute of Natural Sciences, Kyung Hee University, Yongin, 446-701, Republic of Korea}
\email{hyowseo@gmail.com}

%
%

%
\date{March 10, 2021}

\begin{abstract}
A thermoelectric generator converts a temperature difference into electrical energy. Its energy conversion efficiency is determined by the steady-state temperature distribution inside the generator. By assuming the thermoelectric material in the generator has a temperature-independent Seebeck coefficient and the generator is one-dimensional, we show that the second-order integro-differential equation describing the inside temperature distribution has a unique solution for any given ratio of external load resistance to the internal resistance. Hence the efficiency is well defined. Furthermore, we show the efficiency has an explicit formula in terms of the temperature-dependent thermal conductivity and electrical resistivity of the thermoelectric material. On the other hand, if we impose an external load resistance value, not the ratio, then the integro-differential equation can have multiple solutions.
\end{abstract}

 \maketitle
%


\section{Introduction}

Thermoelectric generators convert a temperature difference into electrical energy. They have been used to generate power for spacecraft by converting radioactive decay heat \cite{abelson2006space}. Recently they are attracting attention as a waste heat recovery device \cite{bell2008cooling}. For their utilization, the energy conversion efficiency is required to be enhanced. In this paper, we consider the well-definedness of the energy conversion efficiency and the quantitative relation between the efficiency and the material properties in a one-dimensional thermoelectric generator model. The governing equation is a second-order integro-differential equation with diffusion, reaction and convection terms. However, the mathematical analysis of the efficiency is seldomly found because the efficiency is not defined in the usual heat equations.

Here we introduce the one-dimensional ideal thermoelectric generator model. The generators utilize thermoelectric materials having three properties: thermal conductivity $\kappa$, electrical resistivity $\rho$ and Seebeck coefficient $\alpha$. These thermoelectric material properties are crucial in energy conversion and in general depend on temperature $T$: $\kappa=\kappa(T)$, $\rho=\rho(T)$ and $\alpha=\alpha(T)$. The temperature distribution inside the generators determines the energy conversion efficiency. Assume the generator is of length $L > 0$ and cross-sectional area $A_c > 0$. Then the temperature distribution $T=T(x)$ inside a one-dimensional ideal thermoelectric generator satisfies
\begin{equation}\label{only-te-eqn}
\frac{d}{dx} \Big( \kappa(T) \frac{dT}{dx} \Big) + \rho(T) J[T]^2 - T \frac{d\alpha}{dT}(T) \frac{dT}{dx} J[T] = 0 \quad \text{for $0 < x < L$},
\end{equation}
where the electrical current density $J=J[T]$ is a nonlocal quantity depending on the temperature $T$:
\begin{equation}\label{te-eqn-J-and-R}
J[T] := \frac{V}{R[T] A_c} \quad \text{where} \quad R[T] := \frac{1+\gamma}{A_c} \int_0^L \rho(T(x))\,dx.
\end{equation}
The constant $\gamma \geq 0$ is the ratio of external load resistance to the internal resistance $\frac{1}{A_c} \int_0^L \rho(T(x))\,dx$. The $R[T]$ represents the total electrical resistance, which is the sum of the internal and external resistances. The $V$ is the Seebeck voltage which will be determined by the temperature-dependent Seebeck coefficient and boundary conditions.
For a derivation of the thermoelectric equation \eqref{only-te-eqn}, refer to \cite[Section 1.8.5, Section 2.1.2 and Eq. (2.5)]{goupil2015continuum}.

We assume the thermoelectric generator is under fixed boundary temperatures:
\begin{equation}\label{eq-bc}
T(0) = T_h, \quad T(L) = T_c, \quad T_h \geq T_c > 0,
\end{equation}
where $T_h$ is the hot-side temperature and $T_c$ is the cold-side temperature. Here the strict lower bound 0 of $T_h$, $T_c$ is given to follow the law of thermodynamics that temperature should be bigger than 0 K.

Under the Dirichlet boundary condition \eqref{eq-bc}, the Seebeck voltage $V$ is determined by
\begin{equation}\label{seebeck-voltage}
V = \int_{T_c}^{T_h} \alpha(s)\,ds.
\end{equation}

If the solution $T$ of the thermoelectric equation \eqref{only-te-eqn} is found, the energy conversion efficiency $\eta$ of the ideal thermoelectric generator is
\begin{equation}\label{eq-efficiency}
\eta = \frac{\frac{\gamma}{1+\gamma} V}{\frac{-\kappa(T_h) T_x(0)}{J[T]} + \alpha(T_h) T_h} \quad \text{for $J[T] \not= 0$}
\end{equation}
(see Section \ref{sec-explicit-efficiency-formula} for a derivation). When $J[T]=0$, the efficiency is zero because no electrical energy is generated.
As the efficiency $\eta$ depends on the boundary slope $T_x(0)$, the efficiency \eqref{eq-efficiency} is well defined only when the solution $T$ of the thermoelectric equation \eqref{only-te-eqn} uniquely exists.

If one imposes some smallness assumptions to suppress the effects of the reaction and convection terms in \eqref{only-te-eqn}, one may obtain the unique existence of solutions of \eqref{only-te-eqn}. But without a smallness assumption, the question of the uniqueness is open.

In this paper, we restrict the attention to a special case having a temperature-independent Seebeck coefficient, where the convection term vanishes. Then we can prove the unique existence of the solutions without any smallness assumption. Furthermore, we can find an explicit formula for the efficiency. Our results generalize previous observations on the maximum efficiency for various special forms of $\kappa(T)$ and $\rho(T)$.

\section{Main Result}

We assume the Seebeck coefficient $\alpha(T)$ is temperature-independent, that is, $\alpha(T) \equiv \alpha_0$ for some constant $\alpha_0$. Then the thermoelectric generator model \eqref{only-te-eqn}--\eqref{seebeck-voltage} reduces to
\begin{equation}\label{te-eqn-constant-Seebeck}
\left\{
\begin{aligned}
&(\kappa(T)T_x)_x + \rho(T) J[T]^2 = 0 \quad \text{for $0<x<L$},\\
&J[T] := \frac{V}{R[T] A_c}, \quad R[T] := \frac{1+\gamma}{A_c} \int_0^L \rho(T(x))\,dx,\\
&V := \alpha_0 (T_h-T_c),\\
&T(0) = T_h, ~T(L) = T_c, \quad T_h \geq T_c > 0.
\end{aligned}
\right.
\end{equation}
The primary aim of this paper is to prove the nonlocal reaction-diffusion equation \eqref{te-eqn-constant-Seebeck} has a unique solution under the following positivity and regularity conditions on $\kappa$ and $\rho$.

\begin{Assumption}[Assumption on $\kappa$]
Assume $\kappa \in C^1([T_c, \infty))$ and $\kappa(T) > 0$ for all $T \in [T_c, \infty)$.
\end{Assumption}

\begin{Assumption}[Assumption on $\rho$]
Assume $\rho$ is locally Lipschitz continuous on $[T_c, \infty)$, $\rho(T) > 0$ for all $T \in [T_c, \infty)$, and
\begin{equation}\label{condition-rho-times-kappa}
\int_{T_c}^\infty \rho(T) \kappa(T)\,dT = \infty.
\end{equation}
\end{Assumption}

\begin{Rem}
By the Wiedemann-Franz law (see, e.g., \cite[p.20]{ashcroft1976solid}), for many metals there is a constant $\mathrm{Lo} > 0$ such that
\[ \rho(T) \kappa(T) = \mathrm{Lo} \, T.\]
Hence the condition \eqref{condition-rho-times-kappa} is not restrictive.
\end{Rem}

Under the above assumptions on $\kappa$ and $\rho$, we can prove the unique existence of solutions. Note that no smallness on $\kappa$, $\rho$, and $\alpha_0$ is assumed.

\begin{Thm}[Unique Existence of Temperature Distribution]\label{thm-unique-existence-of-te-eqn}
Suppose the assumptions on $\kappa$ and $\rho$ hold.
Then for any $\alpha_0 \in \R$ and $\gamma \geq 0$, there is a unique classical solution $T \in C^2([0, L])$ of the thermoelectric equation \eqref{te-eqn-constant-Seebeck}.
\end{Thm}

\begin{Rem}
One may formally consider the constant-Seebeck thermoelectric equation \eqref{te-eqn-constant-Seebeck} as a one-dimensional thermistor problem. The thermistor problem describing the Joule heat conduction inside a conductor consists of the charge conservation equation $-\nabla \cdot (\frac{1}{\rho(T)} \nabla \phi) = 0$ and the temperature equation $-\nabla \cdot (\kappa(T) \nabla T) = \frac{1}{\rho(T)} |\nabla \phi|^2$.
In one spatial dimension, assuming $\phi(0) - \phi(L) = \frac{V}{1+\gamma}$, we can reduce the thermistor problem into our thermoelectric equation \eqref{te-eqn-constant-Seebeck}. Let $\tilde{J} := -\frac{1}{\rho(T(x))} \frac{d}{dx}\phi$ for the solution $(\phi, T)$ of the thermistor problem. By the charge conservation equation, $\tilde{J}$ is a constant over $x \in [0, L]$. Integrating both sides of $-\frac{d}{dx}\phi = \tilde{J} \rho(T(x))$ over $x \in [0, L]$, we have $\frac{V}{1+\gamma} = \tilde{J} \int_0^L \rho(T(x))\,dx$. Replacing the $\frac{d}{dx}\phi$ by $-\tilde{J} \rho(T)$ and equating $\tilde{J} = J[T]$, we can check that the solution $T$ of the thermistor problem is a temperature solution of \eqref{te-eqn-constant-Seebeck}.

The unique existence of classical solutions of the thermistor problem has been proved by Cimatti \cite{cimatti1989remark} by transforming the charge conservation equation into the Laplace equation. The Cimatti's result implies our Theorem \ref{thm-unique-existence-of-te-eqn} under a higher regularity of $\kappa, \rho \in C^2([T_c, \infty))$. This higher regularity assumption can be relaxed but at least $C^1$-regularity of $\rho$ is required because a classical solution of the charge conservation equation requires the $C^1$-regularity. On the other hand, our theorem only assumes the local Lipschitz continuity of $\rho$ as we consider the temperature equation directly.
\end{Rem}

Our proof relies on a transformation of the boundary value problem into an initial value problem with an integral constraint. In the uniqueness proof, we  analyze the initial slope of solutions of the initial value problem. Then the analysis on the initial slope allows us to find an explicit formula of the efficiency \eqref{eq-efficiency}. This is our second main result.

\begin{Thm}[Explicit Efficiency Formula]\label{thm-explicit-efficiency-formula}
Suppose that $T_h > T_c > 0$, $\alpha_0 \not=0$ and the assumptions on $\kappa$ and $\rho$ hold.
Then for any load ratio $\gamma \geq 0$, the efficiency $\eta=\eta(\gamma)$ for solutions of \eqref{te-eqn-constant-Seebeck} is well defined and it is explicitly given by
\begin{equation}\label{formula-efficiency}
\eta(\gamma) = \frac{\Delta T}{T_h} \frac{\gamma}{\gamma + 1 + \frac{(\gamma+1)^2}{zT_h} - \frac{1}{2}\frac{\Delta T}{T_h}}
\end{equation}
where $\Delta T := T_h - T_c > 0$ and
\begin{equation}\label{def-z}
z := \frac{\alpha_0^2}{\frac{1}{\Delta T}\int_{T_c}^{T_h}\rho(T) \kappa(T)\,dT}.
\end{equation}

The maximum efficiency is
\begin{equation}\label{eq-maximum-efficiency}
\eta_{\mathrm{max}} := \max_{\gamma \in [0, \infty)} \eta(\gamma) = \frac{\Delta T}{T_h} \frac{\sqrt{1 + z T_m} - 1}{\sqrt{1 + z T_m} + \frac{T_c}{T_h}}
\end{equation}
and it is attained at a unique load ratio
\begin{equation}\label{eq-optimal-load-ratio}
\gamma = \gamma_{\mathrm{opt}} := \sqrt{1 + z T_m}
\end{equation}
where $T_m := \frac{1}{2}(T_h + T_c)$.
\end{Thm}

Our Theorem \ref{thm-explicit-efficiency-formula} shows that the maximum efficiency can be computed by \eqref{eq-maximum-efficiency} \emph{without} solving the integro-differential equation \eqref{te-eqn-constant-Seebeck}. This enables a fast evaluation of thermoelectric generators when the Seebeck coefficients are slowly varying in temperature.

When the thermoelectric material properties $\kappa$, $\rho$, $\alpha$ are all constants, that is, $\kappa \equiv \kappa_0$, $\rho \equiv \rho_0$, and $\alpha = \alpha_0$ for some constants $\kappa_0 > 0$, $\rho_0 > 0$, $\alpha_0 \not= 0$, then the thermoelectric equation \eqref{te-eqn-constant-Seebeck} has an explicit solution in the form of a quadratic polynomial, and the efficiency is given in the same way as \eqref{formula-efficiency} with
\[ z = \frac{\alpha_0^2}{\rho_0 \kappa_0}. \]
As the maximum efficiency \eqref{eq-maximum-efficiency} is monotonically increasing with respect to $z$ \cite{ioffe1957semiconductor}, the $z$ is called the thermoelectric figure of merit.
Our Theorem \ref{thm-explicit-efficiency-formula} naturally generalizes this well-known result for the constant properties model.

Our Theorem \ref{thm-explicit-efficiency-formula} also explains several maximum efficiency formulas found from explicit temperature distributions.  Such examples are given in Section \ref{sec-explicit-efficiency-formula}.

In proving Theorem \ref{thm-explicit-efficiency-formula}, we can also find a necessary and sufficient condition for the strictly decreasing profile of temperature distribution in terms of $z$; see Proposition \ref{prop-decreasing-temperature-criterion}.

On the other hand, if one imposes an external electrical resistance directly, instead of the load ratio, then the unique existence of temperature distribution may fail. In Section \ref{sec-uniqueness-may-fail} we show that for any fixed $\kappa(T)$, there is a $\rho(T)$ such that the uniqueness fails.

\section{Existence and Uniqueness of Solutions}\label{sec-unique-existence}

In this section, we will prove Theorem \ref{thm-unique-existence-of-te-eqn}. The proof of Theorem \ref{thm-unique-existence-of-te-eqn} will be throughout a transformation of the equation.
To proceed, we need an auxiliary function.

\begin{Prop}\label{prop-K-diffeomorphism}
Suppose the assumption on $\kappa$ holds. Let 
\[ K(T) := T_c + \int_{T_c}^T \kappa(s)\,ds \quad \text{and} \quad K_\infty := \sup \{ K(T) : T \in [T_c, \infty)\}.\]
Then $K : [T_c, \infty) \to [T_c, K_\infty)$ is a diffeomorphism. Furthermore, $K, K^{-1} \in C^2([T_c, \infty))$.
\end{Prop}
\begin{proof}
1. $K$ is $C^2([T_c, \infty))$. $K \in C^1([T_c, \infty))$ follows from the theory of Riemann integration because $\kappa(T)$ is continuous on $[T_c,\infty)$ (see, e.g., \cite[Theorem 6.20]{rudin76}). Furthermore an explicit computation shows $K'' = \kappa' \in C([T_c, \infty))$.

2. $K(T) < K_\infty$ for any $T \in [T_c, \infty)$. If $K_\infty = \infty$, the inequality is trivial as $K$ is continuous on $[T_c, \infty)$. Assume that $K_\infty < \infty$ and there is a $T_0 \in [T_c, \infty)$ such that $K(T_0) = K_\infty$. Then the definition of $K_\infty$ and the strict positivity of $\kappa$ imply that $K_\infty \geq K(T_0 + 1) = K(T_0) + \int_{T_0}^{T_0+1} \kappa(s)\,ds > K(T_0) = K_\infty$, which is a contradiction.

3. $K$ is one-to-one. Suppose that for some $T_1, T_2 \in [T_c, \infty)$ it holds that $K(T_1) = K(T_2)$ and $T_1 \geq T_2$. Then $0 = \int_{T_2}^{T_1}\kappa(s)\,ds \geq \kappa_m (T_1 - T_2) \geq 0$ where $\displaystyle \kappa_m := \min_{s \in [T_2, T_1]} \kappa(s) > 0$. Hence $\kappa_m (T_1-T_2) = 0$ so that $T_1 = T_2$.

4. $K$ is onto. Let $\mathcal{R}$ be the range of $K$. Then $K(T_c) = T_c \in \mathcal{R}$. As $K$ is monotonically increasing and $K(T) < K_\infty$ for any $T \in [T_c, \infty)$, $\mathcal{R} \subset [T_c, K_\infty)$. We show the reverse inclusion. Fix a $k_0 \in (T_c, K_\infty)$. Then by the definition of $K_\infty$, there is a $k_1 \in (k_0, K_\infty)$ and $T_1 \in [T_c, \infty)$ such that $k_1 = K(T_1)$. As $K$ is monotonically increasing and $K(T_1) = k_1 > k_0 > T_c = K(T_c)$, we have $T_1 > T_c$. Hence $K(T_c) < k_0 < K(T_1)$ for some $T_1 > T_c$. By the continuity of $K$ and the intermediate value theorem, there is a $T_0 \in (T_c, T_1)$ such that $K(T_0) = k_0$. Therefore $k_0 \in \mathcal{R}$. We have shown that $\mathcal{R} = [T_c, K_\infty)$.

5. $K^{-1}$ is $C^2([T_c, K_\infty))$. As we have shown that $K$ is a bijection, the $K^{-1}$ exists. Next the differentiation formulas $(K^{-1})'(K(T)) = 1/K'(T)= 1/\kappa(T)$ and $(K^{-1})''(K(T)) = - \kappa'(T) / (\kappa(T))^3$ conclude the proof.
\end{proof}

In the assumptions on $\kappa$ and $\rho$, we have considered the region $[T_c, \infty)$ only. This suffices because any solution $T$ of \eqref{te-eqn-constant-Seebeck} satisfies $T \geq T_c$ as Proposition \ref{prop-lower-bound-of-T} shows.

\begin{Lem}\label{lem-lower-bound-of-concave-func}
Let $f \in C([0, L])$ be concave on $[0, L]$. If $f(0) \geq f(L)$, then $f(x) \geq f(L)$ for all $x \in (0,L)$.

Furthermore, if $f(0) > f(L)$, then $f(x) > f(L)$ for all $x \in (0,L)$.
\end{Lem}
\begin{proof}
For any $x \in (0,L)$, let $\alpha := \frac{x}{L} \in (0,1)$ then $x = (1-\alpha) \cdot 0 + \alpha \cdot L$. By the concavity of $f$,
\[ f(x) = f((1-\alpha) \cdot 0 + \alpha \cdot L) \geq (1-\alpha) f(0) + \alpha f(L) \geq f(L) \]
if $f(0) \geq f(L)$. If $f(0) > f(L)$, then the last inequality becomes strict.
\end{proof}

\begin{Prop}\label{prop-lower-bound-of-T}
Let $\kappa \in C^1(\R)$ and $\rho \in C(\R)$ be strictly positive functions on $\R$. If $T \in C^2([0, L])$ is a classical solution of \eqref{te-eqn-constant-Seebeck}, then $T(x) \geq T_c$ for all $x \in [0, L]$.
\end{Prop}
\begin{proof}
From the boundary condition, the inequality holds at $x = 0$ and $x = L$.

From the equation \eqref{te-eqn-constant-Seebeck}, $(K(T))_{xx} = (\kappa(T) T_x)_x = -\rho(T) J[T]^2 \leq 0$ for all $x \in [0, L]$. Hence $K(T) \in C^2([0,L])$ is concave on $[0, L]$. Since $K(T(0)) = K(T_h) \geq K(T_c) = K(T(L))$, by Lemma \ref{lem-lower-bound-of-concave-func},
\[ K(T(x)) \geq K(T(L)) = K(T_c)\]
for all $x \in (0, L)$. Therefore $T(x) \geq T_c$ for all $x \in (0,L)$.
\end{proof}

Suppose that $V = 0$. Then $J[T] = 0$ and the equation \eqref{te-eqn-constant-Seebeck} reduces to $(\kappa(T)T_x)_x = (K(T))_{xx} = 0$ with  the Dirichlet boundary condition $T(0)=T_h$, $T(L)=T_c$. Hence there is a unique classical solution
\[
T(x) = K^{-1} \Big( \frac{K(T_c)-K(T_h)}{L}x + K(T_h) \Big).
\]

Now suppose that $V \not= 0$. Then $T_h > T_c$, and $J[T] \not= 0$ because $\rho(T) > 0$.
For a solution $T \in C^2([0,L])$ of \eqref{te-eqn-constant-Seebeck}, let 
\begin{equation}\label{eqn-y-and-u}
y := |J[T]| \,x \quad \text{and} \quad u(y) := K(T(x)).
\end{equation}
Since
\begin{equation}\label{eqn-u-y-and-u-yy}
u_y=\frac{1}{|J[T]|} \kappa(T)T_x \quad \text{and} \quad u_{yy}= \frac{1}{J[T]^2} (\kappa(T)T_x)_x,
\end{equation}
the equation \eqref{te-eqn-constant-Seebeck} becomes 
\begin{equation}\label{eqn-u}
\left\{
\begin{aligned}
&u_{yy} + \hat{\rho}(u) = 0 \quad \text{for $0<y <y_c$},\\
&u(0)=u_h:=K(T_h), ~u(y_c)=u_c:=K(T_c)=T_c,
\end{aligned}
\right.
\end{equation}
where $y_c = |J[T]|\,L$ and
\[ \hat{\rho} := \rho \circ K^{-1}. \]
In addition, the condition $R[T]=\frac{1+\gamma}{A_c}\int_0^L \rho(T(x))\,dx$ becomes
\begin{equation}\label{integral-constraint-u}
\frac{|V|}{1+\gamma}=\int_0^{y_c} \hat{\rho}(u(y))\,dy.
\end{equation}
Proposition \ref{prop-rho-hat-well-defined} shows that the function $u(y)$ is in the range of $K$ hence $\hat{\rho}(u(y)) = (\rho \circ K^{-1})(u(y))$ is well defined.
\begin{Lem}\label{lem-not-touch-K-infty}
Suppose that $V \not= 0$, $\gamma\geq 0$ and the assumptions on $\kappa$ and $\rho$ hold. For a fixed $y_c > 0$, let $u \in C^2([0,y_c])$ be a function satisfying (i) $u \in [T_c, K_\infty)$ for all $y \in (0, y_c)$, (ii) $u_{yy} + \hat{\rho}(u) = 0$ for all $y \in (0, y_c)$, and (iii) $u(0)=u_h=K(T_h)$.
Then
\[ u(y_c) \in [T_c, K_\infty).\]
\end{Lem}
\begin{proof}
The lower bound of $u(y_c)$ follows from the continuity of $u$.

We prove the upper bound of $u(y_c)$. If $K_\infty = \infty$, the conclusion is trivial. Suppose that $u(y_c) = K_\infty < \infty$. Then there is an increasing sequence $\{y_n\} \subset [0, y_c)$ for all sufficiently large $n$ such that $u(y_n) = K(n)$ and $u(y) \leq K(n)$ for all $y \in [0, y_n]$. By multiplying $-u_y$ on both sides of $u_{yy} + \hat{\rho}(u) = 0$, we have
\[
-\frac{1}{2}\frac{d}{dy} (u_y)^2 = \hat{\rho}(u) u_y.
\]
This equality is well defined for all $y \in [0, y_n]$. Integrating the equality over $[0, y_n]$, we have
\[
\frac{1}{2}(u_y(0)^2 - u_y(y_n)^2) = \int_0^{y_n} \hat{\rho}(u) u_y\,dy.
\]
Note that $u \in C([0, y_n])$, $u([0, y_n]) \subset [T_c, K(n)]$ and $\hat{\rho} \in C([T_c, K(n)])$. Hence we can apply the integration by substitution to the right-hand side:
\[
\int_0^{y_n} \hat{\rho}(u)\,u_y\,dy = \int_{K(T_h)}^{K(n)} \hat{\rho}(u)\,du = \int_{T_h}^{n} \rho(T)\kappa(T)\,dT.
\]
In summary, we have
\[
\frac{1}{2}(u_y(0)^2 - u_y(y_n)^2) = \int_{T_h}^{n} \rho(T)\kappa(T)\,dT.
\]
Taking the limit as $n\to \infty$, we have
\[
\frac{1}{2}(u_y(0)^2 - u_y(y_c)^2) = \int_{T_h}^{\infty} \rho(T)\kappa(T)\,dT = \infty
\]
by the assumption \eqref{condition-rho-times-kappa}. Hence $u_y(y_c)^2 = \infty$ but this is a contradiction.
\end{proof}

\begin{Prop}\label{prop-rho-hat-well-defined}
Suppose that $V \not= 0$, $\gamma\geq 0$ and the assumptions on $\kappa$ and $\rho$ hold. Let $u \in C^2([0,y_c])$ be a solution of \eqref{eqn-u} for some $y_c > 0$ under an extension of $\hat{\rho}$ by $\hat{\rho}(u) \equiv 0$ for all $u \in \R \setminus [T_c, K_\infty)$.
Then $u(y) \in [T_c, K_\infty)$ for all $y \in [0, y_c]$. Hence the extension of $\hat{\rho}$ is not necessary.
\end{Prop}
\begin{proof}
The lower bound follows from Lemma \ref{lem-lower-bound-of-concave-func}.

We prove the upper bound. Suppose that $u(y) \geq K_\infty$ for some $y \in (0, y_c)$. Let $y_\infty := \inf \{ y \in (0, y_c): u(y) \geq K_\infty \}$. As $u(0) = K(T_h) < K_\infty$, $y_\infty > 0$. Then there is a small $\epsilon > 0$ such that $u(y) \in [T_c, K_\infty)$ for all $y \in [y_\infty-\epsilon, y_\infty)$. By Lemma \ref{lem-not-touch-K-infty}, $u(y_\infty) < K_\infty$ but this contradicts the definition of $y_\infty$. Hence $u(y) < K_\infty$ for all $y \in [0, y_c]$.
\end{proof}

If we consider the problem \eqref{eqn-u}, \eqref{integral-constraint-u} for a \emph{given} constant $y_c > 0$, we can revert to the original equation \eqref{te-eqn-constant-Seebeck} as Theorem \ref{thm-one-to-one-local-eqn} shows.

\begin{Thm}\label{thm-one-to-one-local-eqn}
Suppose that $V \not= 0$, $\gamma \geq 0$ and the assumptions on $\kappa$ and $\rho$ hold.
Then there is a one-to-one correspondence between a solution $T \in C^2([0,L])$ of \eqref{te-eqn-constant-Seebeck} and a solution $u$ of the following \emph{local} equation: for some $y_c > 0$, $u \in C^2([0,y_c])$ satisfies
\begin{equation}\label{local-eqn}
\left\{
\begin{aligned}
&u_{yy} + \hat{\rho}(u) = 0 \quad \text{for $0<y<y_c$},\\
&u(0)=u_h:=K(T_h), ~u(y_c) = u_c:=K(T_c),\\
&\int_0^{y_c} \hat{\rho}(u(y))  \,dy = \frac{|V|}{1+\gamma}.
\end{aligned}
\right.
\end{equation}
\end{Thm}
\begin{proof}
1. Let $T\in C^2([0,L])$ be a solution of \eqref{te-eqn-constant-Seebeck}. 
Since $V \not= 0$, the transformation \eqref{eqn-y-and-u} is well defined. Let this correspondence from $T$ to $u$ be $\mathcal{F}$. We have verified that the equation \eqref{integral-constraint-u} is satisfied with $u = \mathcal{F}(T)$.

$\mathcal{F}$ is injective. Suppose that $\mathcal{F}(T_1)=\mathcal{F}(T_2)=u$. 
If $|J[T_1]|=|J[T_2]|$, then $T_1=T_2$ since the spatial scaling is the same and $K$ is invertible.
Suppose that $|J[T_1]|< |J[T_2]|$. Note that $u$ is a concave function because it satisfies $u_{yy} = -\hat{\rho}(u) < 0$.
Hence with $y_c[T] := |J[T]|\, L$,
\[\begin{aligned}
u_c&=u(y_c[T_1])=u\Big(\frac{|J[T_1]|}{|J[T_2]|}y_c[T_2]\Big)=u\Big(\Big( 1-\frac{|J[T_1]|}{|J[T_2]|}\Big)\cdot 0+\frac{|J[T_1]|}{|J[T_2]|}y_c[T_2]\Big)\\
&\geq \Big( 1-\frac{|J[T_1]|}{|J[T_2]|}\Big)u(0)+\frac{|J[T_1]|}{|J[T_2]|}u(y_c[T_2]) = \Big( 1-\frac{|J[T_1]|}{|J[T_2]|}\Big)u_h+\frac{|J[T_1]|}{|J[T_2]|}u_c.
\end{aligned}\]
This implies that $u_c \geq u_h$. Since $u_h \geq u_c$, we have $u_h = u_c$. Then $T_h = T_c$ so that $V = 0$, which is a contradiction.

2. Let $u \in C^2([0,y_c])$ be a solution of \eqref{eqn-u}, \eqref{integral-constraint-u}. Define
\[ x := \frac{L}{y_c} y \quad \text{and} \quad T(x) := K^{-1}(u(y)).\]
By Proposition \ref{prop-rho-hat-well-defined}, $u(y)$ is in the domain of $K^{-1}$ hence this is well defined.
Let $\mathcal{G}$ be this correspondence such that $\mathcal{G}: u \mapsto T$. Then we can check that $\mathcal{G}(u)$ is a solution of \eqref{te-eqn-constant-Seebeck} with
\[ |J[T]| = y_c/L.\]
Also we can check that $\mathcal{F} (\mathcal{G}(u))(y) = u(y)$ for all $0 \leq y \leq y_c$.

$\mathcal{G}$ is injective. Suppose that $\mathcal{G}(u_1)=\mathcal{G}(u_2)=T$. Then $\mathcal{F}(\mathcal{G}(u_1))=\mathcal{F}(\mathcal{G}(u_2))=\mathcal{F}(T)$. Because $\mathcal{F} (\mathcal{G}(u_i))(y) = u_i(y)$ for all $y \in [0, y_c[u_i]]$ and $i=1,2$, we have $u_1(y) = u_2(y)$ for all $y \in [0, \min\{y_c[u_1],y_c[u_2]\}]$. If $y_c[u_1] < y_c[u_2]$, then $u_2(y_c[u_1]) = u_1(y_c[u_1]) = u_c = u_2(y_c[u_2])$. But this contradicts the fact that $u_2$ has a unique $y_c$ such that $u_2(y_c) = u_c$; this follows from Lemma \ref{lem-lower-bound-of-concave-func}.
Hence $y_c[u_1] = y_c[u_2]$ and the injectivity is proved.
\end{proof}

Since a solution of the thermoelectric equation \eqref{te-eqn-constant-Seebeck} corresponds to exactly one solution of \eqref{local-eqn} and vice versa, we focus on the problem \eqref{local-eqn} from now on. Consider the following initial value problem:
\begin{equation}\label{ivp}
\left\{
\begin{aligned}
&u_{yy} + \hat{\rho}(u) = 0 \quad \text{for $y > 0$},\\
&u(0)=u_h:=K(T_h), ~u_y(0) = \theta,
\end{aligned}
\right.
\end{equation}
with a given $\theta \in \R$. By virtue of the local Lipschitz continuity of $\hat{\rho}$, this problem has a unique local solution (see, e.g., \cite[Theorem 2.2]{barbu2016}). Furthermore, the solution can be extended to match the boundary condition $u=u_c:=K(T_c)$ as the following proposition shows.

\begin{Prop}\label{prop-ivp-can-match-bc}
Suppose that $T_c = u_c < u_h < K_\infty$, $\theta \in \R$, and the assumptions on $\kappa$ and $\rho$ hold. Then there is a unique $y_c > 0$ such that there is a solution $u \in C^2([0, y_c])$ of \eqref{ivp} and
\[
u(y_c) = u_c.
\]
\end{Prop}
\begin{proof}
Let $u=u(y)$ be a unique local solution of \eqref{ivp}. Let $y_M > 0$ be the maximal time of existence of $u$. Then $u \in C^2([0, y_M))$. 

Suppose $u(y) > u_c$ for all $y \in [0, y_M)$. By Proposition \ref{prop-rho-hat-well-defined}, $u(y) \in [u_c, K_\infty)$ for all $y \in [0, y_M)$. Then in the same fashion of the proof of Lemma \ref{lem-not-touch-K-infty}, we can show that
\begin{equation}\label{eq-u-y}
u_y(y)^2 = \theta^2 - 2\int_{T_h}^{K^{-1}(u(y))} \rho(T)\kappa(T)\,dT
\end{equation}
for all $y \in [0, y_M)$. Hence $2\int_{T_h}^{K^{-1}(u(y))} \rho(T)\kappa(T)\,dT \leq \theta^2$ for all $y \in [0, y_M)$ so that there is a $K_\theta \in [u_h, K_\infty)$ such that $u(y) \in [u_c, K_\theta]$ for all $y \in [0, y_M)$. Let $\displaystyle \hat{\rho}_m := \min_{u \in [u_c, K_\theta]} \hat{\rho}(u)$. Then $u_{yy} = -\hat{\rho}(u(y)) \leq -\hat{\rho}_m$ for all $y \in [0,y_M)$. Integrating twice this inequality, we have
\begin{equation}\label{ivp-a-priori-bounds}
u_y(y) \leq -\rho_m y + \theta \quad \text{and} \quad u(y) \leq -\frac{\rho_m}{2} y^2 +\theta y +u_h
\end{equation}
for all $y \in [0, y_M)$. If $y_M = \infty$ then the second inequality of \eqref{ivp-a-priori-bounds} shows that $u(y) = u_c$ for some $y \in [0, y_M)$.

Suppose that $y_M < \infty$. Let $\displaystyle \hat{\rho}_M := \max_{u \in [u_c, K_\theta]} \hat{\rho}(u) > 0$.
Then for all $y \in [0, y_M)$,
\[ u_y(y) - \theta = \int_0^y u_{yy}(s)\,ds = \int_0^y -\hat{\rho}(u(s))\,ds \geq -\hat{\rho}_M y \geq -\hat{\rho}_M y_M. \]
With the first inequality in \eqref{ivp-a-priori-bounds},
\[ -\hat{\rho}_M y_M + \theta \leq u_y(y) \leq -\hat{\rho}_m y_M + \theta \quad \text{for all $y \in [0, y_M)$.} \]
In summary, we have shown that $u$, $u_y$ and $u_{yy}$ are bounded in $[0, y_M)$ if $y_M < \infty$.

On the other hand, since $u_y$ and $u_{yy}$ are bounded in $[0, y_M)$, $u$ and $u_y$ are uniformly continuous in $[0, y_M)$. Hence $\displaystyle \lim_{y \uparrow y_M} u(y)$ and $\displaystyle \lim_{y \uparrow y_M} u_y(y)$ exist. By virtue of the boundedness of $u$ and $u_y$, these limits are finite. Then with these limits as an initial condition starting at $y=y_M$, a new initial value problem imposing the same differential equation in \eqref{ivp} is locally solvable. This contradicts the maximality of $y_M$. Therefore, $y_M = \infty$ and there is a $y_c \in [0, \infty)$ such that $u(y_c) = u_c$. Also $y_c > 0$ because $u(0) = u_h > u_c$.

The uniqueness of $y_c$ follows from Lemma \ref{lem-lower-bound-of-concave-func}.
\end{proof}

Let $u(y; \theta)$ be the solution of \eqref{ivp}. Then by Proposition \ref{prop-ivp-can-match-bc}, there is a unique $y_c=y_c(\theta)$ such that $u(y_c(\theta); \theta) = u_c$. Hence the $u(y; \theta)$ is also a solution of \eqref{local-eqn} if the nonlocal constraint \eqref{integral-constraint-u} is satisfied. Let
\[ I(\theta) := \int_0^{y_c(\theta)} \hat{\rho}(u(y; \theta)) \,dy. \]
We will find an analytic formula of $I(\theta)$. By the equation \eqref{ivp},
\begin{equation}\label{relation-I-and_wy}
I(\theta) = \int_0^{y_c(\theta)} -u_{yy}(y;\theta) \,dy = -u_{y}(y_c(\theta);\theta) + \theta.
\end{equation}
The quantity $I(\theta)-\theta = -u_y(y_0(\theta);\theta)$ is symmetric with respect to $\theta = 0$ as the following lemma shows.

\begin{Lem}\label{lem-symmetry-of-I}
Suppose that $T_c = u_c < u_h < K_\infty$, $\theta \in \R$, and the assumptions on $\kappa$ and $\rho$ hold.
Then
\[I(\theta) = I(-\theta) + 2\theta \quad \text{for all $\theta \geq 0$.}\]
\end{Lem}
\begin{proof}
If $\theta = 0$, the statement is trivial.

Let $\theta > 0$ be fixed. Then by the first inequality in \eqref{ivp-a-priori-bounds}, there is a unique $y_p > 0$ such that $u_y(y_p) = 0$. Let $u_1(y) := u(y+y_p)$ and $u_2(y) := u(-y+y_p)$. Then $u_i(0)=u(y_p)$, $u_{i,y}(0) = 0$ and $u_{i,yy} + \hat{\rho}(u_i) = 0$ for $i=1,2$. Since $u_1$ and $u_2$ are solutions of the same initial value problem, $u_1 \equiv u_2$, which implies the $u$ is symmetric with respect to $y=y_p$. Hence $u(2y_p) = u_h$ and $u_y(2y_p) = -\theta$. Therefore
\[y_c(\theta) = 2y_p + y_c(-\theta)\]
and
\[
\begin{split}
I(\theta) &= \int_0^{2y_p + y_c(-\theta)} \hat{\rho}(u(y; \theta))\,dy\\
&= \int_0^{2y_p} \hat{\rho}(u(y; \theta))\,dy + \int_{2y_p}^{2y_p + y_c(-\theta)} \hat{\rho}(u(y; \theta))\,dy \\
&= -\int_0^{2y_p} u_{yy}(y;\theta)\,dy + \int_{0}^{y_c(-\theta)} \hat{\rho}(u(y; -\theta))\,dy \\
&= -(u_y(2y_p;\theta) -u_y(0;\theta)) + I(-\theta) = 2\theta + I(-\theta).
\end{split}
\]
\end{proof}

The following lemma gives an analytic formula of the $I(\theta)$.

\begin{Lem}\label{lem-explicit-I}
Suppose that $T_c = u_c < u_h < K_\infty$, $\theta \in \R$, and the assumptions on $\kappa$ and $\rho$ hold. Then for any $\theta \in \R$,
\begin{equation}\label{explicit-I}
I(\theta) = \theta + \sqrt{\theta^2 + 2\int_{u_c}^{u_h} \hat{\rho}(u)\,du}.
\end{equation}
\end{Lem}
\begin{proof}
First assume that $\theta \leq 0$. Then the solution $u$ of \eqref{ivp} satisfies $u_y(y; \theta) < 0$ for all $y \in (0, y_c(\theta)]$.
From \eqref{eq-u-y},
\[
u_y(y; \theta)^2 = \theta^2 - 2\int_{u_h}^{u(y)} \hat{\rho}(u)\,du
\]
for all $y \in [0, y_c(\theta)]$. In particular,
\[u_y^2(y_c(\theta);\theta) = \theta^2 + 2\int_{u_c}^{u_h} \hat{\rho}(w)\,dw.\]
Since $u_y(y_c(\theta);\theta) < 0$, by \eqref{relation-I-and_wy}, we have \eqref{explicit-I} for all $\theta \leq 0$.

Next assume that $\theta > 0$. By Lemma \ref{lem-symmetry-of-I},
\[
\begin{aligned}
I(\theta) &= I(-\theta) + 2\theta = -\theta + \sqrt{\theta^2 + 2\int_{u_c}^{u_h} \hat{\rho}(u)\,du} + 2 \theta \\
&= \theta + \sqrt{\theta^2 + 2\int_{u_c}^{u_h} \hat{\rho}(u)\,du}.
\end{aligned}
\]
Therefore the identity \eqref{explicit-I} holds for all $\theta \in \R$.
\end{proof}

Finally we can show the unique existence of solutions of the transformed equation \eqref{local-eqn}.

\begin{Thm}[Unique Existence]\label{thm-unique-existence-of-transformed-eqn}
Suppose that $V \not= 0$, $T_c = u_c < u_h < K_\infty$ and the assumptions on $\kappa$ and $\rho$ hold.
Then for any $\gamma \geq 0$, there is a unique solution of the transformed thermoelectric equation \eqref{local-eqn}.
\end{Thm}
\begin{proof}
By Lemma \ref{lem-explicit-I},
\[I'(\theta) = 1 + \frac{\theta}{\sqrt{\theta^2 + 2\int_{u_c}^{u_h} \hat{\rho}(u)\,du}} > 0 \quad \text{for all $\theta \in \R$.}\]
Hence $I(\theta)$ is a continuous and strictly increasing function satisfying $I(-\infty)=0$ and $I(\infty) = \infty$. By the continuity of $I(\theta)$, for any given $\frac{|V|}{1+\gamma} > 0$, there is a \emph{unique} $\theta$ satisfying $I(\theta) = \frac{|V|}{1+\gamma}$, which gives the unique existence of a solution of \eqref{local-eqn}.
\end{proof}

The main uniqueness result, Theorem \ref{thm-unique-existence-of-te-eqn}, follows from Theorem \ref{thm-one-to-one-local-eqn} and Theorem \ref{thm-unique-existence-of-transformed-eqn}.

\section{Explicit Efficiency Formula}\label{sec-explicit-efficiency-formula}

The heat flux $q=q(x)$ inside the thermoelectric generator is
\[
q(x) = -\kappa(T(x)) \frac{dT}{dx} + \alpha(T(x)) T J[T]
\]
(see, e.g., \cite[Section 1.8.5]{goupil2015continuum}). With the heat flux, the thermoelectric equation \eqref{only-te-eqn} can be rewritten as
\[
\frac{dq}{dx} = \rho(T) J[T]^2 + \alpha(T) \frac{dT}{dx} J[T].
\]
Let the hot-side and cold-side heat fluxes be $q_h := q(0)$ and $q_c := q(L)$. Then the efficiency $\eta$ of the thermoelectric generator is
\[
\begin{aligned}
\eta &:= \frac{q_h - q_c}{q_h} = - \frac{1}{q_h} (q(L)-q(0)) \\
&= - \frac{1}{q_h} \int_0^L (\rho(T(x)) J[T]^2 + \alpha(T(x)) \frac{dT}{dx} J[T]) \,dx \\
&= \frac{1}{q_h} \Big( - J[T]^2 \int_0^L \rho(T(x))\,dx + J[T] \int_{T_c}^{T_h} \alpha(s)\,ds \Big) \\
&= \frac{- J[T] \frac{V}{1+\gamma} + J[T] \int_{T_c}^{T_h} \alpha(s)\,ds}{-\kappa(T_h) T_x(0) + \alpha(T_h) T_h J[T]} \quad \text{by \eqref{te-eqn-J-and-R}}.\\
&= \frac{J[T] \frac{\gamma}{1+\gamma} \int_{T_c}^{T_h} \alpha(s)\,ds}{-\kappa(T_h) T_x(0) + \alpha(T_h) T_h J[T]} \quad \text{by \eqref{seebeck-voltage}}.
\end{aligned}
\]
If $J[T] = 0$, the efficiency is zero. Assuming $J[T] \not= 0$, we have
\[
\eta = \frac{\frac{\gamma}{1+\gamma} \int_{T_c}^{T_h} \alpha(s)\,ds}{\frac{-\kappa(T_h) T_x(0)}{J[T]} + \alpha(T_h) T_h}.
\]
In general, the efficiency is determined from the inside temperature distribution $T$. But if the Seebeck coefficient is temperature-independent, we can find an explicit formula for the efficiency \emph{without} finding the $T$, as given in Theorem \ref{thm-explicit-efficiency-formula}.

When $\alpha \equiv \alpha_0$, the efficiency is
\begin{equation}\label{eq-eta-for-constant-alpha}
\eta = \frac{\frac{\gamma}{1+\gamma} \alpha_0 \Delta T}{\frac{-\kappa(T_h) T_x(0)}{J[T]} + \alpha_0 T_h}.
\end{equation}
Here the only unknown term is the relative Fourier heat flux $\frac{-\kappa(T_h) T_x(0)}{J[T]}$. We show this term has an explicit formula.

\begin{Lem}
Suppose that $V \not=0$ and the assumptions on $\kappa$ and $\rho$ hold.
Then for any load ratio $\gamma \geq 0$, the hot-side relative relative Fourier heat flux of the solution $T \in C^2([0, L])$ of \eqref{te-eqn-constant-Seebeck} is explicitly given by
\begin{equation}\label{hot-side-relative-fourier-heat-flux}
\frac{-\kappa(T_h) T_x(0)}{J[T]} = -\frac{1}{2} \frac{|V|}{1+\gamma} + \frac{\int_{T_c}^{T_h} \rho(T) \kappa(T)\,dT}{ \frac{|V|}{1+\gamma} }.
\end{equation}
\end{Lem}
\begin{proof}
By \eqref{eqn-u-y-and-u-yy},
\[
\frac{-\kappa(T(x)) T_x(x)}{J[T]} = -u_y(y)
\]
hence
\[
\frac{-\kappa(T_h) T_x(0)}{J[T]} = -u_y(0) = -\theta^*
\]
where $\theta^*$ is the unique solution of $I(\theta^*) = \frac{|V|}{1+\gamma}$; refer to the proof of Theorem \ref{thm-unique-existence-of-transformed-eqn}. By Lemma \ref{lem-explicit-I}, the $\theta^*$ can be computed explicitly:
\[
\theta^* = \frac{1}{2} \frac{|V|}{1+\gamma} - \frac{\int_{u_c}^{u_h} \hat{\rho}(u)\,du}{ \frac{|V|}{1+\gamma} }.
\]
As $\int_{u_c}^{u_h} \hat{\rho}(u)\,du = \int_{T_c}^{T_h} \rho(T) \kappa(T)\,dT$, the proof is complete.
\end{proof}

Now Theorem \ref{thm-explicit-efficiency-formula} easily follows from the above lemma.
\begin{proof}[Proof of Theorem \ref{thm-explicit-efficiency-formula}]
Put the hot-side relative relative Fourier heat flux \eqref{hot-side-relative-fourier-heat-flux} into the efficiency formula \eqref{eq-eta-for-constant-alpha}. Then by some algebra, we can obtain the efficiency formula \eqref{formula-efficiency}.

Then the equation $\frac{d}{d \gamma}\eta(\gamma) = 0$ implies that $\gamma^2 = 1 + zT_h - \frac{1}{2}z\Delta T = 1 + zT_m$. As $\gamma \geq 0$, this has only one solution $\gamma = \gamma_{\mathrm{opt}} := \sqrt{1 + zT_m}$. The maximum efficiency $\eta_{\mathrm{max}}$ follows from computing $\eta( \gamma_{\mathrm{opt}} )$.
\end{proof}

Our Theorem \ref{thm-explicit-efficiency-formula} enables us to compute the maximum efficiency without solving the integro-differential equation \eqref{te-eqn-constant-Seebeck}. Let us validate the theorem with some efficiency formulas derived from explicit temperature solutions.

\begin{ExQED}[Linear Electrical Resistivity]
Suppose that $\kappa(T) = \kappa_0$, $\rho(T) = \rho_1 T/T_m$, and $\alpha \equiv \alpha_0$ for some constants $\kappa_0, \rho_1 > 0$, $\alpha_0 \not= 0$. Then the maximum efficiency is given by \eqref{eq-maximum-efficiency} by computing \eqref{def-z}:
\begin{equation}\label{eq-z-for-linear-rho}
z = \frac{\alpha_0^2}{\frac{1}{\Delta T}  \int_{T_c}^{T_h} (\rho_1 T/T_m) \kappa_0 \,dT} = \frac{\alpha_0^2 \,\Delta T}{\rho_1 \kappa_0}.
\end{equation}
One can derive the same result by observing that the thermoelectric equation \eqref{te-eqn-constant-Seebeck} has an explicit solution in the form of trigonometric functions; see \cite[Eq. (2.165) and (2.167)]{goupil2015continuum}.
\end{ExQED}

\begin{ExQED}[Linear Thermal Conductivity]
Suppose that $\kappa(T) = \kappa_1 T/T_m$, $\rho(T) = \rho_0$, and $\alpha \equiv \alpha_0$ for some constants $\kappa_1, \rho_0 > 0$, $\alpha_0 \not= 0$. As the $z$ in \eqref{def-z} is invariant under the interchange of $\kappa(T)$ and $\rho(T)$, the $z$ for the maximum efficiency has the same form as \eqref{eq-z-for-linear-rho}:
\[
z = \frac{\alpha_0^2 \,\Delta T}{\rho_0 \kappa_1}.
\]
One can derive the same result by observing that the thermoelectric equation \eqref{te-eqn-constant-Seebeck} has an explicit solution in the form of a square root of quadratic polynomials; see \cite[Section 2.5.7]{goupil2015continuum}.
\end{ExQED}

\begin{ExQED}[When Thermal Conductivity is a Reciprocal Function]
Suppose that $\kappa(T) = \kappa_1/T$, $\rho \equiv \rho_0$, and $\alpha \equiv \alpha_0$ for some constants $\kappa_1 > 0$, $\rho_0 > 0$, $\alpha_0 \not= 0$. In \cite[Eq. (10)]{mahan1991inhomogeneous}, by using an explicit temperature solution in the form of an exponential function, it is observed that the maximum efficiency is given by \eqref{eq-maximum-efficiency} with $z = \frac{\alpha_0^2 \,\Delta T}{\rho_0 \kappa_1 \ln(T_h/T_c)}$. This follows from our formula \eqref{def-z} in Theorem \ref{thm-explicit-efficiency-formula} without considering an explicit temperature solution:
\[
z = \frac{\alpha_0^2}{\frac{1}{\Delta T}  \int_{T_c}^{T_h} \rho_0 \frac{\kappa_1}{T}\,dT} = \frac{\alpha_0^2 \,\Delta T}{\rho_0 \kappa_1 \ln\frac{T_h}{T_c}}.
\]
\end{ExQED}

\begin{ExQED}[When Thermal Conductivity is a Reciprocal Function and Electrical Resistivity is a Logarithmic Function]
Suppose that $\kappa(T) = \kappa_1/T$, $\rho(T) = \rho_0(1 + \rho_1^2 \ln\frac{T}{T_m})$, and $\alpha \equiv \alpha_0$ for some constants $\kappa_1, \rho_0, \rho_1 > 0$, $\alpha_0 \not= 0$. In \cite[Eqs. (48) and (45)]{lee1969influence}, by using an explicit temperature solution in the form of an exponentiated sine function, it is observed that the maximum efficiency is given by \eqref{eq-maximum-efficiency} with
\[
z = \frac{\alpha_0^2 \Delta T}{\rho_0 \kappa_1 (1 + \frac{1}{2}\rho_1^2 \ln\frac{T_h T_c}{T_m^2}) \ln\frac{T_h}{T_c}}.
\]
This follows from our formula \eqref{def-z} in Theorem \ref{thm-explicit-efficiency-formula} by computing that
\begin{align*}
\int_{T_c}^{T_h} \rho(T)\kappa(T)\,dT &= \rho_0 \kappa_1 \int_{T_c}^{T_h} \Big( 1 + \rho_1^2 \ln\frac{T}{T_m} \Big) \frac{1}{T}\,dT =  \rho_0 \kappa_1 \int_{\ln\frac{T_c}{T_m}}^{\ln\frac{T_h}{T_m}} (1+\rho_1^2 s)\,ds\\
&= \rho_0 \kappa_1 \Big( 1 + \frac{1}{2}\rho_1^2 \ln\frac{T_h T_c}{T_m^2} \Big) \ln\frac{T_h}{T_c}.
\end{align*}
\end{ExQED}

Intuition says that if the maximum efficiency is attained, there would be no interior maximum in the temperature distribution. Here we show this is true.

\begin{Prop}\label{prop-decreasing-temperature-criterion}
Suppose that $V \not=0$ and the assumptions on $\kappa$ and $\rho$ hold.
Then for any load ratio $\gamma \geq 0$, the solution $T \in C^2([0, L])$ of \eqref{te-eqn-constant-Seebeck} is strictly decreasing if and only if
\begin{equation}\label{eq-decreasing-criterion}
z \Delta T \leq 2(1+\gamma)^2,
\end{equation}
where $z$ is defined in \eqref{def-z}.
\end{Prop}
\begin{proof}
As the solution $T$ is strictly concave, it suffices to find a necessary and sufficient condition for $T_x(0) \leq 0$.
Observe that $\int_{T_c}^{T_h} \rho(T) \kappa(T)\,dT = \frac{V^2}{z\Delta T}$. By \eqref{hot-side-relative-fourier-heat-flux},
\begin{equation}\label{eq-hot-side-relative-fourier-heat-flux-for-decreasing-criterion}
\frac{-\kappa(T_h) T_x(0)}{J[T]} = -\frac{1}{2} \frac{|V|}{1+\gamma} + \frac{|V|}{z \Delta T} (1+\gamma).
\end{equation}
Hence $T_x(0) \leq 0$ if and only if the right-hand side of \eqref{eq-hot-side-relative-fourier-heat-flux-for-decreasing-criterion} is nonnegative, that is, \eqref{eq-decreasing-criterion} holds.
\end{proof}

\begin{Cor}\label{cor-decreasing-temperature-profile-at-the-maximum-efficiency}
Suppose that $V \not=0$ and the assumptions on $\kappa$ and $\rho$ hold.
Then the solution $T \in C^2([0, L])$ of \eqref{te-eqn-constant-Seebeck} is strictly decreasing when it attains the maximum efficiency.
\end{Cor}
\begin{proof}
By Theorem \ref{thm-explicit-efficiency-formula}, the maximum efficiency is attained when $\gamma = \gamma_{\mathrm{opt}} =  \sqrt{1 + z T_m}$. Next note that $2T_m \geq \Delta T$. Hence
\[
2(1+\gamma_{\mathrm{opt}})^2 > 2\gamma_{\mathrm{opt}}^2 = 2 + 2 zT_m > z \Delta T.
\]
As the criterion \eqref{eq-decreasing-criterion} holds, the solution $T$ is strictly decreasing.
\end{proof}

\begin{ExQED}[Criterion of Decreasing Temperature Profile for Constant Properties]
Suppose that $\kappa \equiv \kappa_0$, $\rho \equiv \rho_0$, and $\alpha \equiv \alpha_0$ for some constants $\kappa_0 > 0$, $\rho_0 > 0$, $\alpha_0 \not= 0$. In \cite[Eq. (21)]{chung2014nonlocal} it is observed that the solution $T \in C^2([0, L])$ of \eqref{te-eqn-constant-Seebeck} is strictly decreasing if and only if $\frac{\alpha_0^2}{\rho_0 \kappa_0} \leq \frac{2(1+\gamma)^2}{\Delta T}$. This follows immediately from \eqref{eq-decreasing-criterion}.
\end{ExQED}

\begin{ExQED}[Criterion of Decreasing Temperature Profile for Linear Electrical Resistivities]
Suppose that $\kappa \equiv \kappa_0$, $\rho(T) = \rho_1 T$, and $\alpha \equiv \alpha_0$ for some constants $\kappa_0 > 0$, $\rho_1 > 0$, $\alpha_0 > 0$. In \cite[Eq. (38)]{chung2014nonlocal} it is observed that the solution $T \in C^2([0, L])$ of \eqref{te-eqn-constant-Seebeck} is strictly decreasing if and only if
\begin{equation}\label{eq-decreasing-criterion-for-linear-rho}
\frac{T_c}{T_h} \leq \cos \Big(\sqrt{ \frac{\rho_1}{\kappa_0} } J[T] L \Big).
\end{equation}
Here we show that this condition follows from \eqref{eq-decreasing-criterion}.

Let $a := \sqrt{ \frac{\rho_1}{\kappa_0} } J[T]$ for brevity. Then the solution is explicitly given by $T(x) = T_h \cos(ax) + \frac{T_c-T_h\cos(aL)}{\sin(aL)} \sin(ax)$. From the uniqueness of solutions, we also have $0 < aL < \pi$. Using the trigonometric identity $\frac{1-\cos\theta}{\sin\theta} = \frac{\sin\theta}{1+\cos\theta}$, we can compute that
\[
\frac{1}{1+\gamma} = \frac{J[T]}{V} \int_0^L \rho(T(x))\,dx = \frac{ 2T_m \sqrt{\rho_1 \kappa_0} }{V} \frac{\sin(aL)}{1+\cos(aL)}.
\]
On the other hand, one can easily compute that $z = \frac{\alpha_0^2}{\rho_1 \kappa_0 T_m}$. Put these two identities into \eqref{eq-decreasing-criterion}, then we have $\frac{\Delta T}{2T_m} \geq \big( \frac{\sin(aL)}{1+\cos(aL)} \big)^2 = \big( \tan(\frac{1}{2}aL) \big)^2$. This is equivalent to $2\arctan \big( \sqrt{ \frac{\Delta T}{2T_m} } \big) \geq aL$ because $0 < aL < \pi$. Taking cosine on both sides (note that taking cosine reverses the inequality) and using the identity $\cos(2\arctan(\theta)) = \frac{1-\theta^2}{1+\theta^2}$, we have
\[
\cos(aL) \geq \cos\Big(2\arctan\big(\sqrt{ \frac{\Delta T}{2T_m} }\big)\Big) = \frac{T_c}{T_h},
\]
which is \eqref{eq-decreasing-criterion-for-linear-rho}.
\end{ExQED}

\begin{ExQED}[Relation between the Relative Fourier Heat Flux and the Maximum Efficiency]
Suppose that $\alpha \equiv \alpha_0$ for some constant $\alpha_0 \not= 0$. For fixed functions of $\kappa(T)$ and $\rho(T)$, let us assume that a temperature solution $T_{\mathrm{opt}}$ of the thermoelectric equation \eqref{only-te-eqn}--\eqref{seebeck-voltage} attains the maximum efficiency $\eta_{\mathrm{max}}$, and that $T_{\mathrm{opt}}(x)$ is strictly decreasing with respect to $x$. Under these assumptions, Sherman et al. \cite[Below Eq. (57)]{sherman1960calculation} \emph{formally} showed that the hot-side relative Fourier heat flux can be written in terms of the maximum efficiency as follows:
\begin{equation}\label{eq-relative-fourier-heat-flux-at-maximum-efficiency}
\frac{-\kappa(T_h) \frac{d}{dx}T_{\mathrm{opt}}(0)}{J[T_{\mathrm{opt}}]} = \frac{1 - \eta_{\mathrm{max}}}{ \sqrt{1 - (1 - \eta_{\mathrm{max}})^2} } \sqrt{ 2 \int_{T_c}^{T_h} \rho(T)\kappa(T)\,dT }.
\end{equation}

Under the assumption that $\alpha \equiv \alpha_0$ for some constant $\alpha_0 \not=0$, and the assumptions in Theorem \ref{thm-explicit-efficiency-formula}, we can rigorously prove the above relation \eqref{eq-relative-fourier-heat-flux-at-maximum-efficiency}. For brevity, let $r := \int_{T_c}^{T_h} \rho(T)\kappa(T)\,dT$. Using \eqref{hot-side-relative-fourier-heat-flux}, \eqref{eq-optimal-load-ratio} and the relation $|V| = \sqrt{r z \Delta T}$, we have
\begin{equation}\label{eq-relative-fourier-heat-flux-in-terms-of-gamma-zDeltaT}
\begin{aligned}
\frac{-\kappa(T_h) \frac{d}{dx}T_{\mathrm{opt}}(0)}{J[T_{\mathrm{opt}}]} &= -\frac{1}{2} \frac{|V|}{1+\gamma_{\mathrm{opt}}} + \frac{r}{ \frac{|V|}{1+\gamma_{\mathrm{opt}}} } \\
&= \frac{1}{2} \Big( -\frac{\sqrt{z \frac{\Delta T}{2} }}{1+\gamma_{\mathrm{opt}}} + \frac{ 1+\gamma_{\mathrm{opt}} }{ \sqrt{z \frac{\Delta T}{2}} } \Big) \sqrt{2r}.
\end{aligned}
\end{equation}
Next observe that by \eqref{eq-maximum-efficiency} and $T_c = T_m - \frac{\Delta T}{2}$,
\begin{align*}
\frac{1}{\eta_{ \mathrm{max} }} - 1 &= \frac{1}{\Delta T} \frac{ T_h\gamma_{\mathrm{opt}} + T_m - \frac{\Delta T}{2} }{ \gamma_{\mathrm{opt}}-1 } -1 = \frac{1}{\Delta T} \frac{ (T_m - \frac{\Delta T}{2})\gamma_{\mathrm{opt}} + T_m + \frac{\Delta T}{2} }{ \gamma_{\mathrm{opt}}-1 } \\
&= \frac{1}{\Delta T} \frac{ (T_m - \frac{\Delta T}{2})\gamma_{\mathrm{opt}} + T_m + \frac{\Delta T}{2} }{ \gamma_{\mathrm{opt}}-1 } \cdot \frac{ \gamma_{\mathrm{opt}}+1}{\gamma_{\mathrm{opt}}+1} \\
&= \frac{1}{\Delta T} \frac{ (T_m-\frac{\Delta T}{2})\gamma_{\mathrm{opt}}^2 + 2T_m \gamma_{\mathrm{opt}} + (T_m+\frac{\Delta T}{2}) }{ zT_m } \\
&= \frac{1}{2} \frac{1}{z \frac{\Delta T}{2}} \Big(  (1 + \gamma_{\mathrm{opt}})^2 - z \frac{\Delta T}{2} \Big).
\end{align*}
Similarly, we have
\[
\frac{2}{\eta_{ \mathrm{max} }} - 1 = \frac{1}{z \frac{\Delta T}{2}} (1 + \gamma_{\mathrm{opt}})^2.
\]
Therefore the right-hand side of \eqref{eq-relative-fourier-heat-flux-at-maximum-efficiency} is
\begin{align*}
\frac{1 - \eta_{\mathrm{max}}}{ \sqrt{1 - (1 - \eta_{\mathrm{max}})^2} } \sqrt{2r} &= \frac{ \frac{1}{\eta_{\mathrm{max}}} - 1 }{ \sqrt{ (\frac{1}{ \eta_{\mathrm{max}} })^2 - ( \frac{1}{\eta_{\mathrm{max}}} - 1)^2} } \sqrt{2r} = \frac{ \frac{1}{\eta_{\mathrm{max}}} - 1 }{ \sqrt{ \frac{2}{ \eta_{\mathrm{max}} }-1 } } \sqrt{2r} \\
&= \frac{1}{2}\Big( - \frac{ \sqrt{z \frac{\Delta T}{2}} }{ 1+\gamma_{\mathrm{opt}} }  + \frac{ 1+\gamma_{\mathrm{opt}} }{ \sqrt{z \frac{\Delta T}{2}} } \Big)  \sqrt{2r},
\end{align*}
which is equal to \eqref{eq-relative-fourier-heat-flux-in-terms-of-gamma-zDeltaT}. Therefore the relation \eqref{eq-relative-fourier-heat-flux-at-maximum-efficiency} holds. Note that the Sherman et al.'s assumption on the strictly decreasing profile of $T_{\mathrm{opt}}$ is also true by Corollary \ref{cor-decreasing-temperature-profile-at-the-maximum-efficiency}.
\end{ExQED}

\section{Uniqueness may fail under a given load resistance}\label{sec-uniqueness-may-fail}

A direct way to consider an external load resistance is to specify the load resistance $R_{\mathrm{load}}$ rather than the load ratio $\gamma$. If we impose $R_{\mathrm{load}}$ instead of $\gamma$, the nonlocal quantity $R[T]$ in thermoelectric equation \eqref{te-eqn-constant-Seebeck} should be modified while all the other formulations in \eqref{te-eqn-constant-Seebeck} remain the same:
\begin{equation}\label{te-eqn-load-resistance}
\left\{
\begin{aligned}
&(\kappa(T)T_x)_x + \Big(\frac{V}{R[T] A_c}\Big)^2 \rho(T) = 0 \quad \text{for $0<x<L$},\\
&R[T] := \frac{1}{A_c} \int_0^L \rho(T(x))\,dx + R_{\mathrm{load}},\\
&V := \alpha_0 (T_h-T_c),\\
&T(0) = T_h, ~T(L) = T_c, \quad T_h \geq T_c \geq 0.
\end{aligned}
\right.
\end{equation}
Following the similar lines of the arguments in Section \ref{sec-unique-existence}, we can find an equivalent equation:
\begin{Thm}
Suppose that $V \not= 0$ and the assumptions on $\kappa$ and $\rho$ hold.
Then there is a one-to-one correspondence between a solution $T \in C^2([0,L])$ of \eqref{te-eqn-load-resistance} and a solution $u$ of the following \emph{local} equation: for some $y_c > 0$, $u \in C^2([0,y_c])$ satisfies
\begin{equation}\label{local-eqn-load-resistance}
\left\{
\begin{aligned}
&u_{yy} + \hat{\rho}(u) = 0 \quad \text{for $0<y<y_c$},\\
&u(0)=u_h := K(T_h), ~u(y_c) = u_c := K(T_c),\\
&\int_0^{y_c} (\hat{\rho}(u(y)) + S_{\mathrm{load}}) \,dy = |V|,
\end{aligned}
\right.
\end{equation}
where $\hat{\rho}(u) := (\rho \circ K^{-1})(u)$ and $S_{\mathrm{load}} := R_{\mathrm{load}}A_c/L$. The $K$ is defined in Proposition \ref{prop-K-diffeomorphism}.
\end{Thm}
Hence the nonuniqueness of solutions of \eqref{local-eqn-load-resistance} yields the nonuniqueness of solutions of \eqref{te-eqn-load-resistance}.

\begin{Thm}[Nonuniqueness]\label{nonuniqueness-thm}
Let $T_h > T_c \geq 0$, $L > 0$, $A_c > 0$ be given constants and suppose that $\kappa$ be a given function satisfying the assumption on $\kappa$.  Then there are a function $\rho$ satisfying the assumption on $\rho$, and constants $R_{\mathrm{load}} > 0$, $\alpha_0 \not= 0$ such that the equation \eqref{te-eqn-load-resistance} has two different classical solutions.
\end{Thm}
\begin{proof}
We construct two different solutions of \eqref{local-eqn-load-resistance}. For a given $\hat{\rho}_h > 0$, choose the following function as $\hat{\rho}$:
\[
\hat{\rho}(u) :=
\begin{cases}
M(u-u_h) + \hat{\rho}_h, \quad \text{for $u \geq u_h$,}\\
\hat{\rho}_h, \quad \text{for $u < u_h$.}
\end{cases}
\]
Here $M$ is a positive constant which will be determined later.
The above $\hat{\rho}$ is uniformly Lipschitz continuous on $[0, \infty)$ and has a positive lower bound $\hat{\rho}_h$.
Let $u(y; \theta)$ be the solution of the initial value problem
\[
\left\{
\begin{aligned}
&u_{yy} + \hat{\rho}(u) = 0 \quad \text{for $y > 0$},\\
&u(0)=u_h, ~u_y(0) = \theta,
\end{aligned}
\right.
\]
then by the same arguments in the previous section, we can show that for any $\theta \in \R$, there is a unique $y_c=y_c(\theta) >0$ such that $u \in C^2([0, y_c])$ and $u(y_c(\theta); \theta) = u_c$. This $u$ is a solution of \eqref{local-eqn-load-resistance} if it satisfies the nonlocal constraint in \eqref{local-eqn-load-resistance}. The nonlocal constraint can be rewritten by
\[ H(\theta) :=  I(\theta) + S_{\mathrm{load}} \,y_c(\theta) = |V|,\]
where the $I(\theta)$ is the same one given in Lemma \ref{lem-explicit-I}.

Suppose $\theta > 0$. Then $u(y; -\theta) = -\frac{\hat{\rho}_h}{2}y^2 -\theta y + u_h$ so that
\[ y_c(-\theta) = \frac{1}{\hat{\rho}_h} (-\theta + \sqrt{\theta^2 + 2 \hat{\rho}_h \Delta u}), \quad \text{where $\Delta u := u_h - u_c$.}\]
On the other hand, by the same arguments in the previous section, there is a unique $y_0 = y_0(\theta) > 0$ such that $u_y(y_0(\theta); \theta) = 0$. Furthermore $y_c(\theta) = 2 y_0(\theta) + y_c(-\theta)$.
In $y \in [0, 2y_0(\theta)]$, the $u(y; \theta)$ satisfies $u \geq u_h$, hence it is explicitly given by
\[u(y; \theta) = u_h + \frac{\hat{\rho}_h}{M} (\cos(\sqrt{M}y) - 1) + \frac{\theta}{\sqrt{M}} \sin(\sqrt{M}y).\]
Because $y_0(\theta)$ satisfies $u_y(y_0(\theta); \theta) = -\frac{\hat{\rho}_h}{\sqrt{M}} \sin(\sqrt{M}y) + \theta \cos(\sqrt{M}y) = 0$,
we have
\[y_0(\theta) = \frac{1}{\sqrt{M}} \arctan \Big( \frac{\sqrt{M}}{\hat{\rho}_h} \theta \Big).\]
Therefore by the explicit formula of $I(\theta)$ in Lemma \ref{lem-explicit-I}, we have for any $\theta > 0$,
\begin{align}
H(\theta) & = \theta + \sqrt{\theta^2 + 2\hat{\rho}_h \Delta u} \nonumber \\
& \quad + S_{\mathrm{load}} \Big( \frac{2}{\sqrt{M}} \arctan \Big( \frac{\sqrt{M}}{\hat{\rho}_h} \theta \Big)  + \frac{1}{\hat{\rho}_h} (-\theta + \sqrt{\theta^2 + 2 \hat{\rho}_h \Delta u}) \Big) \nonumber \\
& = \Big( 1- \frac{S_{\mathrm{load}}}{\hat{\rho}_h} \Big)\theta +  \Big( 1+ \frac{S_{\mathrm{load}}}{\hat{\rho}_h} \Big) \sqrt{\theta^2 + 2\hat{\rho}_h \Delta u} + \frac{2 S_{\mathrm{load}}}{\sqrt{M}} \arctan \Big( \frac{\sqrt{M}}{\hat{\rho}_h} \theta \Big). \label{eq-H-of-theta}
\end{align}
Hence
\[
H'(\theta) = 1- \frac{S_{\mathrm{load}}}{\hat{\rho}_h} + \Big( 1+ \frac{S_{\mathrm{load}}}{\hat{\rho}_h} \Big) \frac{\theta}{\sqrt{\theta^2 + 2\hat{\rho}_h \Delta u}} + 2 \frac{S_{\mathrm{load}}}{\hat{\rho}_h} \frac{1}{1 + M \theta^2 / \hat{\rho}_h^2}
\]

Choose a small $\theta_1$ and a large $M$ such that
\begin{equation}\label{nonuniqueness-choice-of-theta1-and-M}
\frac{\theta_1}{\sqrt{\theta_1^2 + 2\hat{\rho}_h \Delta u}} = \frac{1}{2} \quad \text{and} \quad M \theta_1^2/\hat{\rho}_h^2 = 16.
\end{equation}
Then
\[H'(\theta_1) = \frac{3}{2} - \frac{13}{34} \frac{S_{\mathrm{load}}}{\hat{\rho}_h} < 0\]
if
\begin{equation}\label{nonuniqueness-load-assumption}
\frac{S_{\mathrm{load}}}{\hat{\rho}_h} > \frac{51}{13}.
\end{equation}
As $H'(\infty) = 2 > 0$, under the assumptions \eqref{nonuniqueness-choice-of-theta1-and-M} and \eqref{nonuniqueness-load-assumption}, there is a $\theta_2 > \theta_1$ such that
\[
H(\theta_2) = H(\theta_1).
\]
Therefore by defining $|V| := H(\theta_1)$, both $u(y; \theta_1)$ and $u(y; \theta_2)$ satisfy the nonlocal condition $H(\theta_1)=H(\theta_2)=|V|$. As we have found two different solutions of \eqref{local-eqn-load-resistance}, $u(y; \theta_1)$ and $u(y; \theta_2)$, the proof is complete. Note that the assumptions \eqref{nonuniqueness-choice-of-theta1-and-M} and \eqref{nonuniqueness-load-assumption} are not optimal, and can be generalized with a small $\theta_1$, a large $M$, and a large $S_{\mathrm{load}}$.
\end{proof}

\begin{Rem}
A nonuniqueness for the thermistor problem is observed by Cimatti \cite{cimatti1990stationary} when $T_h = T_c$ and an external battery and an external resistor is attached to a conductor. Our nonuniqueness theorem, Theorem \ref{nonuniqueness-thm}, applies when $T_h > T_c$ and the voltage is generated inside a conductor having an external resistor but no external battery.
\end{Rem}

Following the proof of Theorem \ref{nonuniqueness-thm} with specified parameter values, we can find an explicit example of the nonuniqueness.

\begin{Ex}[Three Solutions]\label{ex-three-solutions}
Let $\kappa \equiv 1$, $T_h = 2$, $T_c = A_c = L = 1$.
Now define $M:=48$,
\[
\rho(T) :=
\begin{cases}
M(T-T_h) + T_h, \quad \text{for $T \geq T_h$,}\\
T_h, \quad \text{for $T < T_h$.}
\end{cases}
\]
$R_{\mathrm{load}} := 8$ and $\alpha_0 := -\frac{3}{2}\sqrt{3} + \frac{5}{2}\sqrt{19} + \frac{4}{\sqrt{3}} \arctan(3) \approx 11.18$. Then the equation \eqref{te-eqn-load-resistance} has three classical solutions
\begin{equation}\label{eq-multiple-solutions}
T_i(x) :=
\begin{cases}
\frac{\rho(T_h)}{M}( \cos(\sqrt{M}y_{c,i} x) - 1) + \frac{\theta_i}{\sqrt{M}} \sin(\sqrt{M}y_{c,i} x) + T_h, \quad \text{for $0 \leq x \leq x_{0,i}$,} \\
-y_{c,i}^2 (x-x_{0,i})^2 - \theta_i y_{c,i} (x-x_{0,i}) + T_h, \quad \text{for $x_{0,i} < x \leq 1$,}
\end{cases}
\end{equation}
for $i=0,1,2$ where
\begin{align*}
x_{0,i} &:= \frac{2 \arctan( \sqrt{M}\theta_i / \rho(T_h) )}{ \sqrt{M} y_{c,i} }, \\
y_{c,i} &:= \frac{2}{\sqrt{M}} \arctan(\sqrt{M} \theta_i / \rho(T_h)) + \frac{1}{\rho(T_h)} (-\theta_i + \sqrt{ \theta_i^2+ 2\rho(T_h)\Delta T }), \\
\theta_1 &:= \frac{\sqrt{3}}{2} \approx 0.866,
\end{align*}
and $\theta_0$ and $\theta_2$ are the unique real numbers satisfying $\theta_0 < \theta_ 1 < \theta_2$ and
\[
\Big(1 - \frac{R_{\mathrm{load}}}{\rho(T_h)}\Big) \theta_i + \Big( 1 + \frac{R_{\mathrm{load}}}{\rho(T_h)} \Big) \sqrt{\theta_i^2 + 2 \rho(T_h) \Delta T} + \frac{2 R_{\mathrm{load}}}{\sqrt{M}} \arctan \Big( \frac{\sqrt{M}\theta_i}{\rho(T_h)} \Big) = \alpha_0 \Delta T,
\]
that is,
\[
-3\theta_i + 5\sqrt{\theta_i^2 + 4} + \frac{4}{\sqrt{3}} \arctan(2\sqrt{3} \,\theta_i) = \alpha_0.
\]
A numerical computation shows that $\theta_0 \approx 0.402$ and $\theta_2 \approx 1.483$.

We can check that $R[T_0] = \frac{\alpha_0}{y_{c,0}} \approx 10.24$, $R[T_1] = \frac{\alpha_0}{y_{c,1}} \approx 10.99$ and $R[T_2] = \frac{\alpha_0}{y_{c,2}} \approx 12.41$ so that the load resistance ratios of the three solutions are different. The difference is also expected by the uniqueness theorem, Theorem \ref{thm-unique-existence-of-te-eqn}, because if they were the same,  the solutions should be the same.

\begin{figure}
\hfill
\subfigure[Three Solutions in Example \ref{ex-three-solutions}]{\includegraphics[width=.4\linewidth]{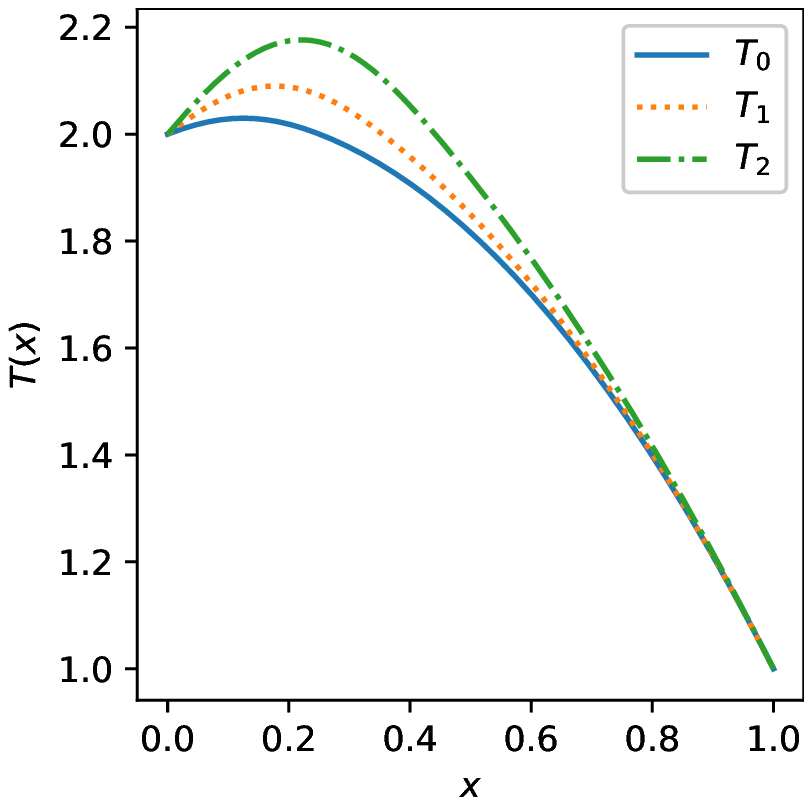}}
\hfill
\subfigure[$H(\theta)$ and $\theta_i$'s in Example \ref{ex-three-solutions}]{\includegraphics[width=0.4\linewidth]{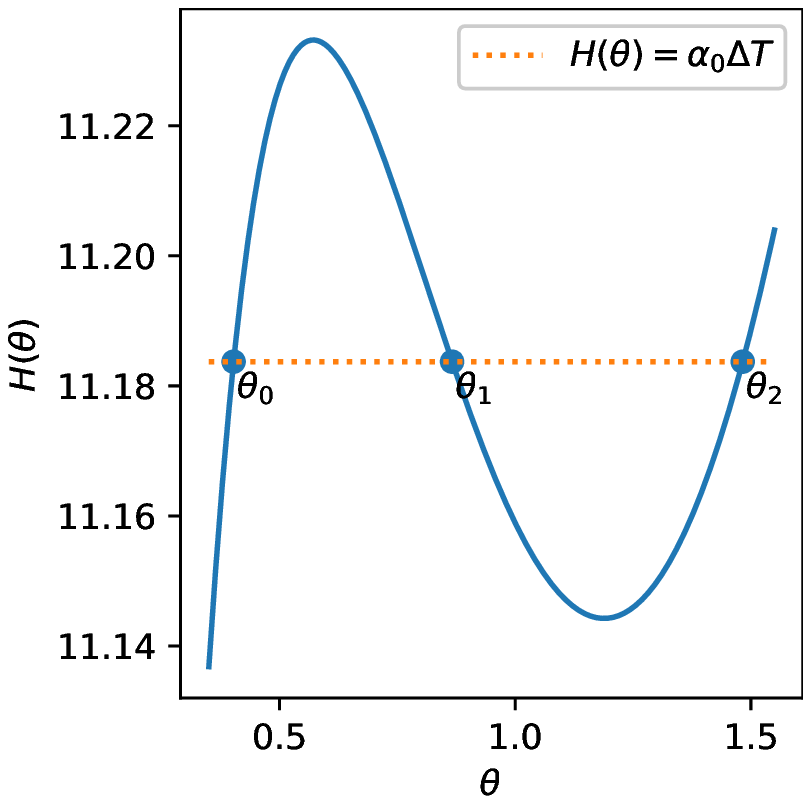} \label{fig-H-of-theta-for-three-solutions}}
\hfill
\caption{There are three solutions in Example \ref{ex-three-solutions}. The number of the solutions is equal to the number of solutions of the equation $H(\theta) = \alpha_0 \Delta T$ as we have shown in the proof of Theorem \ref{nonuniqueness-thm}. The $H(\theta)$ is given in \eqref{eq-H-of-theta}. The Figure \ref{fig-H-of-theta-for-three-solutions} suggests that there are no more than three solutions.}
\end{figure}
\end{Ex}

\begin{Ex}[Two Solutions]\label{ex-two-solutions}
The Figure \ref{fig-H-of-theta-for-three-solutions} shows that there are only two solutions for an appropriate value of $\alpha_0$, that is, there are only two solutions of $H(\theta) = \alpha_0 \Delta T$ for some $\alpha_0 > 0$. Let $\theta_1$ be the unique number satisfying $\theta_1 > 1$ and $H'(\theta) = 0$. A numerical computation shows that $\theta_1 \approx 1.189$. And let $\alpha_0 := H(\theta_1)/\Delta T \approx 11.14$ and keep all the other parameters the same as in Example \ref{ex-three-solutions}. Then there is a $\theta_0 \approx 0.357$ such that $0 < \theta_0 < \theta_1$ and $H(\theta_0) = \alpha_0 \Delta T$. Hence the equation \eqref{te-eqn-load-resistance} has two classical solutions $T_0$ and $T_1$ which are given by \eqref{eq-multiple-solutions}. For the solutions, $R[T_0] \approx 10.18$ and $R[T_1] \approx 11.69$.

\begin{figure}
\hfill
\subfigure[Two Solutions in Example \ref{ex-two-solutions}]{\includegraphics[width=.4\linewidth]{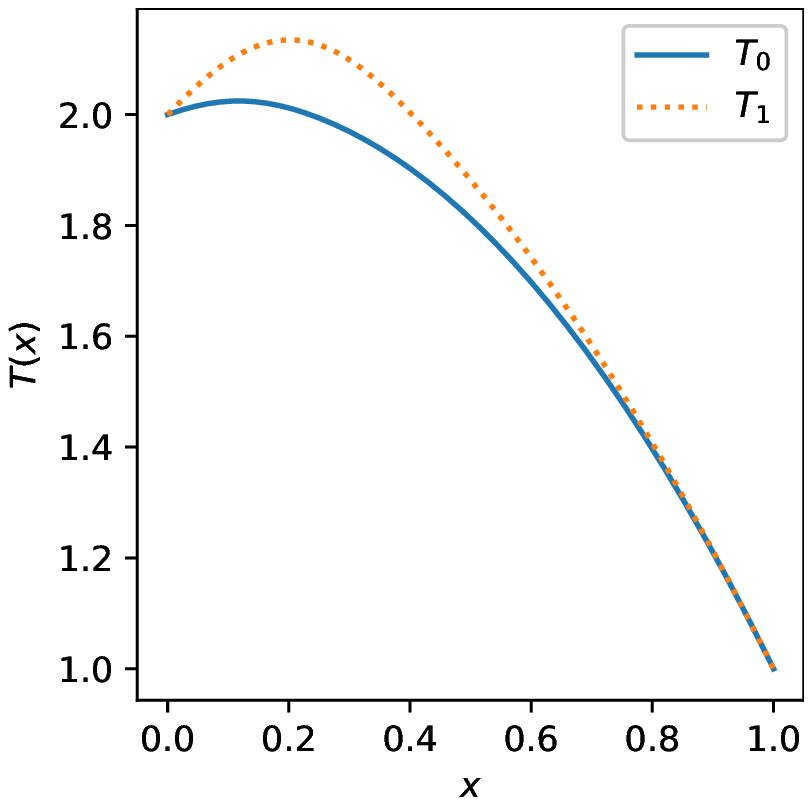}}
\hfill
\subfigure[$H(\theta)$ and $\theta_i$'s in Example \ref{ex-two-solutions}]{\includegraphics[width=0.4\linewidth]{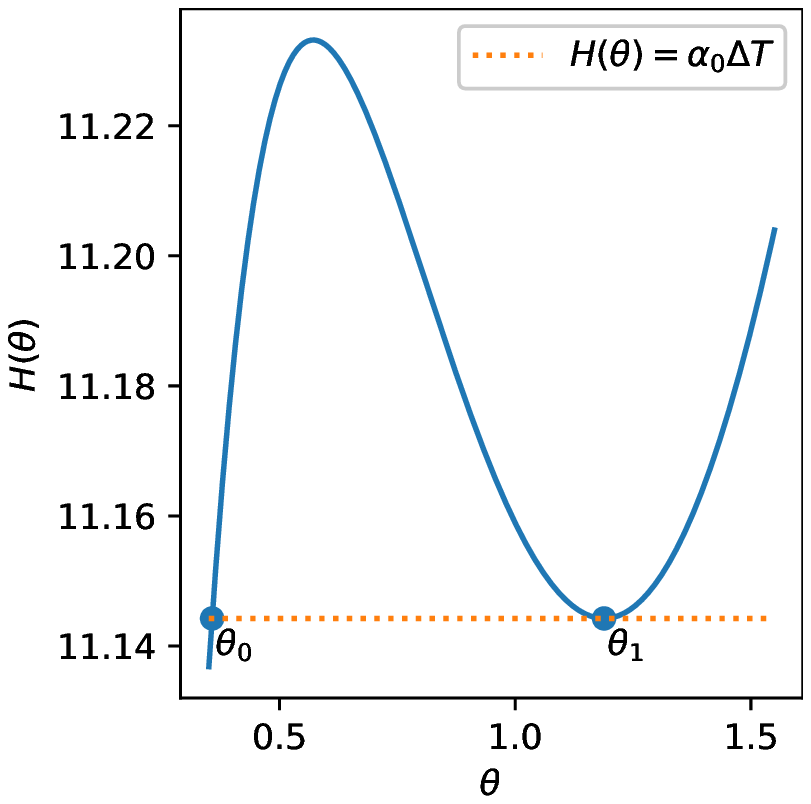}}
\hfill
\caption{There are two solutions in Example \ref{ex-two-solutions}.}
\end{figure}
\end{Ex}

\section*{Acknowledgement}
J.C. and B.R. were supported by the Korea Electrotechnology Research Institute (KERI) Primary Research Program through the National Research Council of Science and Technology (NST) funded by the Ministry of Science and ICT (MSIT) of the Republic of Korea: Grant No. 21A01003.
H.S. was supported by NRF of Korea (no. 2020R1I1A1A01069585).

\bibliographystyle{amsplain}
\bibliography{te_uniqueness_constant_seebeck}

\end{document}